\documentclass[reqno,12pt]{amsart}
\usepackage{amsfonts}
\usepackage{bbm}
\usepackage{graphicx,amsmath}
\usepackage{} %leqno is the option to put formula numbers on the left side
\numberwithin{equation}{section}
\usepackage{indentfirst}
\usepackage{color}
\usepackage{amssymb}
\usepackage{mathrsfs}
\usepackage[shortcuts]{extdash}
\usepackage{xy}
\usepackage{tikz}
\usepackage{extpfeil} % 提供可扩展的箭头
\usepackage{amsmath}
\usepackage{mathtools}
\usepackage{tikz-cd}
\xyoption{all}
 \def\Hom{\mbox{\rm Hom}}  \def\Infl{\mbox{\rm Infl}\,}\def\Defl{\mbox{\rm Defl}\,}\def\Cone{\mbox{\rm Cone}\,}

\def\A{\mathcal{A}\,}

\theoremstyle{plain} %text of this environment is typesetted in italics
\newtheorem{theorem}{\bf Theorem}[section]
\newtheorem{lemma}[theorem]{\bf Lemma}
\newtheorem{corollary}[theorem]{\bf Corollary}

\theoremstyle{definition} %text of this environment is typesetted in roman letters
\newtheorem{definition}[theorem]{\bf Definition}

\newtheorem{remark}[theorem]{\bf Remark}
\newtheorem{example}[theorem]{\bf Example}

\newcommand{\bt}{\begin{theorem}}
\newcommand{\et}{\end{theorem}}
\newcommand{\bl}{\begin{lemma}}
\newcommand{\el}{\end{lemma}}
\newcommand{\bd}{\begin{definition}}
\newcommand{\ed}{\end{definition}}
\newcommand{\bc}{\begin{corollary}}
\newcommand{\ec}{\end{corollary}}
\newcommand{\bp}{\begin{proof}}
\newcommand{\ep}{\end{proof}}
\newcommand{\bx}{\begin{example}}
\newcommand{\ex}{\end{example}}
\newcommand{\br}{\begin{remark}}
\newcommand{\er}{\end{remark}}
\newcommand{\be}{\begin{equation}}
\newcommand{\ee}{\end{equation}}
\newcommand{\ba}{\begin{align}}
\newcommand{\ea}{\end{align}}
\newcommand{\bn}{\begin{enumerate}}
\newcommand{\en}{\end{enumerate}}
\newcommand{\bcs}{\begin{cases}}
\newcommand{\ecs}{\end{cases}}

\makeatletter
\renewcommand{\section}{\@startsection{section}{1}{0mm}
  {-\baselineskip}{0.5\baselineskip}{\bf\leftline}}
\makeatother

\begin{document}

\title[Extriangulated factorization systems, $s$-torsion pairs and recollements]{Extriangulated factorization systems,\\ $s$-torsion pairs and recollements}

\author[Y. Xu, H. Zhang, Z. Zhu]{Yan Xu, Haicheng Zhang, Zhiwei Zhu}

\address{Ministry of Education Key Laboratory of NSLSCS, School of Mathematical Sciences, Nanjing Normal University,
Nanjing 210023, P.R.~China}
\email{swgfeng@outlook.com (Xu)}

\email{zhanghc@njnu.edu.cn (Zhang)}

\email{1985933219@qq.com (Zhu)}

%%%%%%%%%%%%%%% footnote %%%%%%%%%%%%%%%%
%\subjclass[2010]{17B37, 16G20, 17B20.}

\keywords{Extriangulated categories; $s$-torsion pairs; factorization systems; recollements.}
%\thanks{$*$~Corresponding author.}
%%%%%%%%%%%% Authors addresses %%%%%%%%%%%%%

%%%%%%%%%%%%%%%%%%%%%%%%%%%%%%%%%%%%%%%%%

\begin{abstract}
 We introduce extriangulated factorization systems in extriangulated categories and show that there exists a bijection between $s$-torsion pairs and extriangulated factorization systems. We also consider the gluing of $s$-torsion pairs and extriangulated factorization systems under recollements of extriangulated categories.
\end{abstract}

\maketitle

\section{Introduction}
 Extriangulated categories were introduced by Nakaoka and Palu \cite{8}, which can be viewed as a simultaneous generalisation of triangulated categories
and exact categories. Exact categories and extension-closed subcategories
 of triangulated categories are extriangulated categories, and there are also some examples
 of extriangulated categories which may be neither exact categories
 nor triangulated categories. Extriangulated categories provide a unified framework for some research topics in exact
 categories and triangulated categories.

Factorization systems play a significant role in modern category theory, primarily due to their ability to equip categories with sophisticated structural properties.
Formally, a classical factorization system in a category $\mathscr{C}$ consists of two classes $\mathcal{E}$ and $\mathcal{M}$ of morphisms satisfying the factorization axiom and orthogonality condition (c.f. \cite{9,13}).
We recall that a morphism $f$ in $\mathscr{C}$ is left orthogonal to another morphism $g$ (or $g$ is right orthogonal to $f$) if for any commutative square of solid arrows
$$\xymatrix{
A \ar[d]_{f} \ar[r] &C \ar[d]^{g}\\
B \ar[r]   \ar@{-->}[ur]^{d}        &D
}$$
there is a unique morphism $d$ such that the two above triangles commute. If such $d$ exists, but may not be unique, one calls it weak orthogonality.
The weak factorization systems can be defined by the factorization axiom and weak orthogonality condition. A class $\mathcal{X}$ of morphisms of the category $\mathscr{C}$ is said to have the {\em 3-for-2 property}, it means that for any two composable morphisms $\xymatrix{\cdot\ar[r]^-{g}&\cdot\ar[r]^-{f}&\cdot}$ in $\mathscr{C}$, if two of the morphisms $f,g,f\circ g$ belong to $\mathcal{X}$, so does the third. A factorization system $(\mathcal{E},\mathcal{M})$ is called a {\em torsion theory} if both $\mathcal{E}$ and $\mathcal{M}$ have the 3-for-2 property. Given a torsion theory $(\mathcal{E},\mathcal{M})$, set
\[
0/\mathcal{E} := \left\{ X \in \mathscr{C} \mid
\scalebox{0.6}{$\begin{bmatrix} 0 \\ \downarrow \\X \end{bmatrix}$} \in \mathcal{E} \right\}
\]
and
\[
\mathcal{M}/0 := \left\{ X \in \mathscr{C} \mid
\scalebox{0.6}{$\begin{bmatrix} X \\ \downarrow \\0 \end{bmatrix}$} \in \mathcal{M} \right\}.
\]

For a torsion theory $(\mathcal{E},\mathcal{M})$ in an abelian category $\mathscr{C}$,  it is called {\em normal} if for any object $X\in\mathscr{C}$, taking the factorization
$$\xymatrix{X\ar[r]^-{e}& X'\ar[r]^-m&0}$$
of the morphism $X\rightarrow 0$ and the pullback of $e$ along the morphism $0\rightarrow X'$
$$\xymatrix{T\ar[r]\ar[d]&X\ar[d]^-e\\0\ar[r]&X'}$$
we have that $T\rightarrow 0$ belongs to $\mathcal{E}$ (c.f. \cite{13}).
Rosick\'{y} and Tholen \cite{9} showed that the following assignment
$$(\mathcal{E},\mathcal{M})\longmapsto (0/\mathcal{E}, \mathcal{M}/0)$$
gives a bijection between classical torsion pairs and normal torsion theories.

Switching to triangulated categories, the role played by classical torsion pairs in abelian categories is played by $t$-structures. Loregian and Virili  \cite{13} introduced triangulated factorization systems, and then generalised the above bijection to a bijection between $t$-structures and triangulated factorization systems.
In order to provide a general framework for the studies of $t$-structures in triangulated categories and torsion
pairs in abelian categories, Adachi, Enomoto and Tsukamoto \cite{2} introduced the notions of $s$-torsion pairs in extriangulated categories with negative first extensions.
Moreover, when such extriangulated categories are triangulated, they provided a bijection between $s$-torsion pairs and $t$-structures.

Recollements of triangulated categories were introduced by Be{\u\i}linson,
Bernstein and Deligne in \cite{bd}. The recollements of abelian categories first appeared in the construction of the
category of perverse sheaves on a singular space in \cite{mv}. Recently, Wang, Wei, and Zhang \cite{16} defined recollements of extriangulated categories, which provide a unified generalization of recollements of abelian categories and recollements of triangulated categories.

The main aim of this paper is to provide a unified framework for the bijections in \cite{9} and \cite{13} by defining extriangulated factorization systems and applying $s$-torsion pairs.
The paper is organized as follows: we summarize some basic definitions and properties of extriangulated categories, and give the definitions of $s$-torsion pairs and extriangulated factorization systems in Section 2. In Section 3, we establish a bijection between $s$-torsion pairs and extriangulated factorization systems. Section 4 is devoted to gluing $s$-torsion pairs and extriangulated factorization systems under recollements of extriangulated categories.

Throughout the paper, we assume that all categories are additive categories, all subcategories are full and closed under isomorphisms. For a category $\mathscr{C}$, we denote by $\mathscr{C}(X,Y)$ the set of morphisms from $X$ to $Y$ in $\mathscr{C}$. For a finite-dimensional algebra $A$, we denote by $K^b({\rm proj}A)$ the homotopy category of bounded complexes over projective $A$-modules, and by $D^b({\rm mod}A)$ the bounded derived category of $A$.

\section{Preliminaries}

\subsection{Extriangulated categories}
Let us recall some notions concerning extriangulated categories from \cite{8}.
Let $\mathbb{E}:\mathscr{C}^{op}\times\mathscr{C}\rightarrow Ab$ be a biadditive functor, where $Ab$ is the category of abelian groups. For any pair of objects $A,C\in\mathscr{C}$, an element $\delta\in\mathbb{E}(C,A)$ is called an
$\mathbb{E}$-{\em extension}. The zero element $0\in\mathbb{E}(C,A)$ is called the {\em split} $\mathbb{E}$-{\em extension}. For any morphism $a\in\mathscr{C}(A,A')$ and $c\in\mathscr{C}(C',C)$, we have the following $\mathbb{E}$-extensions
$$\mathbb{E}(C,a)(\delta)\in\mathbb{E}(C,A')~\text{and}~\mathbb{E}(c,A)(\delta)\in\mathbb{E}(C',A),$$
which are denoted by $a_*\delta$ and $c^*\delta$, respectively.
A morphism $(a,c):\delta\rightarrow \delta'$ of $\mathbb{E}$-extensions $\delta\in\mathbb{E}(C,A)$ and $\delta'\in\mathbb{E}(C',A')$ is a pair of morphisms $a\in\mathscr{C}(A,A')$ and $c\in\mathscr{C}(C,C')$ such that $a_*\delta=c^*\delta'$.
Two sequences of morphisms $A\stackrel{x}{\longrightarrow}B\stackrel{y}{\longrightarrow}C$ and $A\stackrel{x'}{\longrightarrow}B'\stackrel{y'}{\longrightarrow}C$ in $\mathscr{C}$ are said to be {\em equivalent} if there exists an isomorphism $b\in\mathscr{C}(B,B')$ such that the following diagram is commutative
$$\xymatrix{A \ar[r]^x \ar@{=}[d]& B\ar[r]^y \ar[d]^b_{\simeq}&C\ar@{=}[d]\\
A\ar[r]^{x'}&B'\ar[r]^{y'}&C.}$$
We denote the equivalence class of $A\stackrel{x}{\longrightarrow}B\stackrel{y}{\longrightarrow}C$ by $[A\stackrel{x}{\longrightarrow}B\stackrel{y}{\longrightarrow}C]$, and for any $A,C\in\mathscr{C}$, denote as
$$0=[\xymatrix{A\ar[r]^{\tiny{\begin{pmatrix}1\\0\end{pmatrix}}\quad\quad}&A\oplus B\ar[r]^{\quad\tiny{\begin{pmatrix}0 & 1\end{pmatrix}}}&C}].
$$

\begin{definition}
  \cite[Definition 2.9]{8} Let $\mathfrak{s}$ be a correspondence, which associates an equivalence class $\mathfrak{s}(\delta)=[A\stackrel{x}{\longrightarrow}B\stackrel{y}{\longrightarrow}C]$ to each $\mathbb{E}$-extension $\delta\in\mathbb{E}(C,A)$. This $\mathfrak{s}$ is called a {\em realization} of $\mathbb{E}$ if for any morphism $(a,c):\delta\rightarrow \delta'$ with $\mathfrak{s}(\delta)=[A\stackrel{x}{\longrightarrow}B\stackrel{y}{\longrightarrow}C]$ and $\mathfrak{s}(\delta')=[A'\stackrel{x'}{\longrightarrow}B'\stackrel{y'}{\longrightarrow}C']$, there is a commutative diagram as follows:
  $$\xymatrix{A \ar[r]^x \ar[d]^a& B\ar[r]^y \ar[d]^b&C\ar[d]^c\\
  A'\ar[r]^{x'}&B'\ar[r]^{y'}&C'.}
  $$
  A realization $\mathfrak{s}$ of $\mathbb{E}$ is said to be {\em additive} if the following conditions are satisfied:

  \;(a) For any $A,C\in\mathscr{C}$, the split $\mathbb{E}$-extension $0\in\mathbb{E}(C,A)$ satisfies $\mathfrak{s}(0)=0$.

  \;(b) $\mathfrak{s}(\delta\oplus \delta')=\mathfrak{s}(\delta)\oplus\mathfrak{s}(\delta')$ for any pair of $\mathbb{E}$-extensions $\delta$ and $\delta'$.
\end{definition}

An extriangulated category is a triple $(\mathscr{C},\mathbb{E},\mathfrak{s})$ consisting of the following data satisfying
certain axioms:

$(1)$ $\mathscr{C}$ is an additive category and $\mathbb{E}:\mathscr{C}^{op}\times\mathscr{C}\rightarrow Ab$ is a biadditive functor.

$(2)$ $\mathfrak{s}$ is an additive realization of $\mathbb{E}$, which defines the class of conflations satisfying the
axioms (ET1)-(ET4), (ET3)$^{\rm op}$ and (ET4)$^{\rm op}$ (see \cite[Definition 2.12]{8} for details).

A sequence $A\stackrel{x}{\longrightarrow}B\stackrel{y}{\longrightarrow}C$ is called a {\em conflation} if it realizes some $\mathbb{E}$-extension $\delta\in\mathbb{E}(C,A)$. Then $x$ is called an {\em inflation} and $y$ is called a {\em deflation}, and $A\stackrel{x}{\longrightarrow}B\stackrel{y}{\longrightarrow}C\stackrel{\delta}{\dashrightarrow}$ is called an $\mathbb{E}$-{\em triangle}.
For an $\mathbb{E}$-triangle $A\stackrel{x}{\longrightarrow}B\stackrel{y}{\longrightarrow}C\stackrel{\delta}{\dashrightarrow}$, we denote $A={\rm cocone}(y)$ and $C={\rm cone}(x)$. An $\mathbb{E}$-triangle is {\em split} if it realizes $0$.

A subcategory $\mathcal{U}$ of an extriangulated category $\mathscr{C}$ is {\em closed under extensions} if for any conflation $A\stackrel{x}{\longrightarrow}B\stackrel{y}{\longrightarrow}C$ with $A,C\in\mathcal{U}$, we have $B\in\mathcal{U}$.
An object $P$ in $\mathscr{C}$ is called {\em projective} if for any conflation $A\stackrel{x}{\longrightarrow}B\stackrel{y}{\longrightarrow}C$ and any morphism $c\in\mathscr{C}(P,C)$, there exists a morphism $b\in\mathscr{C}(P,B)$ such that $yb=c$. The notion of {\em injective} objects can be defined dually. We denote by $\mathcal{P}(\mathscr{C})$ the full subcategory of projective objects in $\mathscr{C}$.

Let $\mathscr{C}$ be an extriangulated category. For any subcategories $\mathcal{T}, \mathcal{F}$ of $\mathscr{C}$, denote by $\mathcal{T} * \mathcal{F}$ the subcategory of $\mathscr{C}$ which consists of the objects $X$ admitting an $\mathbb{E}$-triangle $T{\longrightarrow}X{\longrightarrow}F{\dashrightarrow}$ with $T\in\mathcal{T}$ and $F\in\mathcal{F}$. For any subcategory $\mathcal{U}$, define
$$\Infl\mathcal{U}=\{f~\text{is~an~inflation}\;|\;{\rm cone}(f)\in\mathcal{U}\}$$
and
   $$\Defl\mathcal{U}=\{f~\text{is~a~deflation}\;|\;{\rm cocone}(f)\in\mathcal{U}\}.$$
   Given a class $\mathcal{I}$ of inflations in $\mathscr{C}$, set
\begin{align*}
{\rm Cone}\,\mathcal{I}=\{U\in\mathscr{C}\ |\ U\cong \operatorname{cone}(f)\ \textrm{for~some~inflation}\ f\in\mathcal{I}\}.
\end{align*}
Given a class $\mathcal{J}$ of deflations in $\mathscr{C}$, set
\begin{align*}
{\rm Cocone}\,\mathcal{J}=\{U\in\mathscr{C}\ |\ U\cong \operatorname{cocone}(f)\ \textrm{for~some~deflation}\ f\in\mathcal{J}\}.
\end{align*}

\subsection{$s$-torsion pairs in extriangulated categories}
In order to give the definition of $s$-torsion pairs, we need to recall negative first extensions
in extriangulated categories.
\begin{definition}\cite[Definition 2.3]{2}
  An extriangulated category $\mathscr{C}$ is called an {\em extriangulated category with negative first extensions}, if $\mathscr{C}$ has a negative first extension structure which consists of the following data:
  \begin{enumerate}
    \item $\mathbb{E}^{-1}:\mathscr{C}^{\mathrm{op}} \times \mathscr{C} \rightarrow Ab$ is an additive bifunctor.
    \item For each \(\delta \in \mathbb{E}(C, A)\), there exist two natural transformations
\[\delta_{\sharp }^{-1} : \mathbb{E}^{-1}(-, C) \to \mathscr{C}(-, A),\]
\[\delta_{-1}^{\sharp} : \mathbb{E}^{-1}(A, -) \to \mathscr{C}(C, -)\]
  \end{enumerate}
such that for each  $\mathbb{E}$-triangle $A\stackrel{f}{\longrightarrow}B\stackrel{g}{\longrightarrow}C\stackrel{\delta}{\dashrightarrow}$
and $W \in \mathscr{C}$, we have the following two exact sequences
\begin{align*}
&\mathbb{E}^{-1}(W,A)
\xrightarrow{\mathbb{E}^{-1}(W,f)}
\mathbb{E}^{-1}(W,B)
\xrightarrow{\mathbb{E}^{-1}(W,g)}
\mathbb{E}^{-1}(W,C)
\xrightarrow{(\delta_{\sharp}^{-1})_W}
\mathscr{C}(W,A)
\xrightarrow{\mathscr{C}(W,f)}
\mathscr{C}(W,B),\\
&\mathbb{E}^{-1}(C,W)
\xrightarrow{\mathbb{E}^{-1}(g,W)}
\mathbb{E}^{-1}(B,W)
\xrightarrow{\mathbb{E}^{-1}(f,W)}
\mathbb{E}^{-1}(A,W)
\xrightarrow{(\delta_{-1}^{\sharp})_W}
\mathscr{C}(C,W)
\xrightarrow{\mathscr{C}(g,W)}
\mathscr{C}(B,W).
\end{align*}
\end{definition}
We can naturally regard triangulated categories and exact
categories as extriangulated categories with negative first extensions (c.f.\cite[Example 2.4]{2}).

\begin{definition}\cite[Definition 3.1]{2}
Let $\mathscr{C}$ be an extriangulated category with negative first extensions. A pair $(\mathcal{T}, \mathcal{F})$ of subcategories of $\mathscr{C}$ is called an \emph{s-torsion pair} in $\mathscr{C}$ if it satisfies the following conditions:

$(1)$~$\mathscr{C}= \mathcal{T} * \mathcal{F}$.\quad  $(2)$~$\mathscr{C}(\mathcal{T}, \mathcal{F}) = 0$.\quad $(3)$ $\mathbb{E}^{-1}(\mathcal{T}, \mathcal{F}) = 0$.
\end{definition}

\subsection{Extriangulated factorization systems}
In what follows, we always assume, unless otherwise stated, that $\mathscr{C}$ is an extriangulated category with negative first extensions.
In this subsection, we define certain factorization systems in $\mathscr{C}$.
\begin{definition}
Given two inflations $l: A \to B$ and $r: C \to D$ in $\mathscr{C}$, we say that {\em $l$ is left orthogonal to $r$} (or {\em $r$ is right orthogonal to $l$}), denoted by $l \bot r$, if $$\mathscr{C}(\text{\rm{cone}}(l),\text{\rm{cone}}(r))=0~\text{and}~\mathbb{E}^{-1}(\text{\rm{cone}}(l),\text{\rm{cone}}(r))=0.$$

For a class $\mathcal{I}$ of inflations in $\mathscr{C}$, define
 \[ \mathcal{I}^\perp = \left\{f~\text{is~an~inflation}~|~g \perp f~\text{for~any~inflation}~g\in \mathcal{I}\right\}\]
 and
\[ {}^ \perp \mathcal{I}= \left\{ f~\text{is~an~inflation}~|~f \perp g~\text{for~any~inflation}~g\in \mathcal{I}\right\}.\]
\end{definition}

\begin{definition}
Let $\mathcal{L}$ and $\mathcal{R}$ be two classes of inflations in $\mathscr{C}$. We call $\left( {\mathcal{L},\mathcal{R}} \right)$ an {\em inflation factorization system} if the following conditions are satisfied:
\begin{enumerate}
  \item For each inflation $f$ in $\mathscr{C}$, there is a factorization $f=rl$ with $l \in \mathcal{L}$ and $r \in \mathcal{R}$.
  \item $\mathcal{L}={}^ \perp \mathcal{R}$.
  \item $\mathcal{L}^\perp=\mathcal{R}$.
\end{enumerate}
The {\em orthogonality of deflations} and {\em deflation factorization systems} can be dually defined. The inflation factorization systems and deflation factorization systems are collectively referred to as {\em extriangulated factorization systems}.
\end{definition}

\begin{remark}
$(1)$~If $\mathscr{C}$ is an abelian category, for any two inflations $l: A \to B$ and $r: C \to D$ in $\mathscr{C}$, they are both monomorphisms. In this case, $l \bot r$ if and only if $$\mathscr{C}({\rm Coker}(l),{\rm Coker}(r))=0.$$ Then for any commutative top square of solid arrows,
$$\xymatrix@C=2pc{
A \ar[d]_{l} \ar[r] &C \ar[d]^{r}\\
B \ar[d] \ar[r]   \ar@{-->}[ur]^{d}        &D \ar[d]\\
\text{\rm{Coker}}(l) \ar[r]^{0} &\text{\rm{Coker}}(r)
}$$
the zero morphism is the unique morphism such that the bottom square is commutative, if and only if
there exists a unique morphism $d$ such that the two triangles in the above commutative top square commute. Hence, the orthogonality defined here implies the orthogonality in classical factorization systems, and so do the orthogonality of deflations.

$(2)$~If $\mathscr{C}$ is a triangulated category, any morphism in $\mathscr{C}$ is an inflation, and for any morphisms $l$ and $r$, if $l \bot r$, then $l$ is left homotopy orthogonal to $r$ in the sense of \cite[Definition 1.1]{13}.
\end{remark}

\section{Bijections between $s$-torsion pairs and factorization systems}
In this section, we establish a bijection between $s$-torsion pairs and extriangulated factorization systems.
\begin{theorem}\label{main}
Let $\mathscr{C}$ be an extriangulated category with negative first extensions. Then the assignments
\[
(\mathcal{T}, \mathcal{F}) \mapsto (\Infl\mathcal{T},\Infl\mathcal{F}) \quad \text{and} \quad (\mathcal{L}, \mathcal{R}) \mapsto ({\rm Cone}(\mathcal{L}), {\rm Cone}(\mathcal{R}))
\]
give the mutually inverse bijections between the following classes:
\begin{enumerate}
  \item $s$-torsion pairs in $\mathscr{C}$;
  \item Inflation factorization systems in $\mathscr{C}$.
\end{enumerate}
\end{theorem}
\begin{proof}
Let $f:X\rightarrow Y$ be any inflation in $\mathscr{C}$, and
let $X\stackrel{f}{\longrightarrow}Y{\longrightarrow}C{\dashrightarrow}$ be an $\mathbb{E}$-triangle. Suppose that $(\mathcal{T}, \mathcal{F})$ is an $s$-torsion pair in $\mathscr{C}$, then there is an $\mathbb{E}$-triangle $T{\longrightarrow}C{\longrightarrow}F{\dashrightarrow}$ with $T \in \mathcal{T}$ and  $F \in \mathcal{F}$. By $(\rm{ET4})^{op}$, we obtain a commutative diagram of conflations
$$\xymatrix{X \ar[r]^l \ar@{=}[d]& K \ar[r] \ar[d]^r & T\ar[d]\\
  X\ar[r]^{f}&Y\ar[r]\ar[d]&C\ar[d]\\
  &F\ar@{=}[r]&F.}
  $$
Thus, we obtain that $f=rl$ with $l\in \Infl\mathcal{T}$ and $r\in \Infl\mathcal{F}$.

By the definition of $s$-torsion pairs, clearly, $\Infl\mathcal{T}\subseteq{}^ \perp \Infl\mathcal{F}$. Conversely, let $h \in {}^ \perp \Infl\mathcal{F}$, since for any $F \in \mathcal{F}$, there is an $\mathbb{E}$-triangle $0{\longrightarrow}F\stackrel{{\rm id}_F}{\longrightarrow}F{\dashrightarrow}$, we obtain that the zero morphism $0 \to F$ belongs to $\Infl\mathcal{F}$. Hence, we have that $\mathscr{C}(\text{\rm{cone}}(h),F)=0$. By \cite[Proposition 3.2]{2}, we obtain that ${\rm cone}(h)\in\mathcal{T}$, i.e., $h \in \Infl\mathcal{T}$, and then ${}^ \perp\Infl\mathcal{F} \subseteq\Infl\mathcal{T}$. Thus, $\Infl\mathcal{T}={}^ \perp \Infl\mathcal{F}$. Similarly, it can be proved that
$(\Infl\mathcal{T})^\perp=\Infl\mathcal{F}$.
Therefore, $(\Infl\mathcal{T}, \Infl\mathcal{F})$ is an inflation factorization system.

On the other hand, let $\left( {\mathcal{L},\mathcal{R}} \right)$ be an inflation factorization system, for any object $X \in \mathscr{C}$, consider the factorization of the inflation $0\rightarrow X$ with respective to $\left( {\mathcal{L},\mathcal{R}} \right)$, we have the following commutative diagram
$$\xymatrix{ 0 \ar[rr]\ar[rd]& &X\\
&T_X\ar[ru]_-r&
}$$
with $T_X \in {\rm Cone}(\mathcal{L})$ and $r \in \mathcal{R}$. Thus, we have an $\mathbb{E}$-triangle $T_X{\longrightarrow}X{\longrightarrow}F_X{\dashrightarrow}$ with $F_X \in {\rm Cone}(\mathcal{R}).$ Thus, $\mathscr{C}={\rm Cone}(\mathcal{L})\ast{\rm Cone}(\mathcal{R})$. While
$$\mathbb{E}^{-1}({\rm Cone}(\mathcal{L}),{\rm Cone}(\mathcal{R}))=\mathscr{C}({\rm Cone}(\mathcal{L}),{\rm Cone}(\mathcal{R}))=0$$ is immediately obtained from $\mathcal{L}\perp \mathcal{R}$. Hence,  $(\text{Cone}(\mathcal{L}),\text{Cone}(\mathcal{R}))$ is an $s$-torsion pair.

Clearly, ${\rm Cone}(\Infl\mathcal{T})\subseteq \mathcal{T}$. For any $T \in \mathcal{T}$, by the $\mathbb{E}$-triangle $0{\longrightarrow}T\stackrel{{\rm id}_T}{\longrightarrow}T{\dashrightarrow}$, we have that $T \in {\rm Cone}(\Infl\mathcal{T})$, i.e., $\mathcal{T}\subseteq{\rm Cone}(\Infl\mathcal{T})$.
Thus, ${\rm Cone}(\Infl\mathcal{T})= \mathcal{T}$.
Moreover, clearly, $\mathcal{L}\subseteq \Infl({\rm Cone}\mathcal{L})$, and $\Infl({\rm Cone}\mathcal{L})\subseteq\mathcal{L}$ is obtained from $\mathcal{L}={}^ \perp \mathcal{R}$. Therefore, we complete the proof.
\end{proof}

Dually, we have the following theorem, and omit the proof.
\begin{theorem}
Let $\mathscr{C}$ be an extriangulated category with negative first extensions. Then the assignments
\[
(\mathcal{T}, \mathcal{F}) \mapsto (\Defl\mathcal{T},\Defl\mathcal{F}) \quad \text{and} \quad (\mathcal{L}, \mathcal{R}) \mapsto ({\rm Cocone}(\mathcal{L}), {\rm Cocone}(\mathcal{R}))
\]
give the mutually inverse bijections between the following classes:
\begin{enumerate}
  \item $s$-torsion pairs in $\mathscr{C}$;
  \item Deflation factorization systems in $\mathscr{C}$.
\end{enumerate}
\end{theorem}

Let $\mathscr{C}$ be an abelian category, then each inflation in $\mathscr{C}$ is a monomorphism, and each deflation is an epimorphism. So the inflation factorization systems
and deflation factorization systems may also called {\em monomorphism factorization systems} and {\em epimorphism factorization systems}, respectively. Moreover, in this case, the $s$-torsion pairs are just the classical torsion pairs.
\begin{corollary}
Let $\mathscr{C}$ be an abelian category. Then assignments
\[
(\mathcal{T}, \mathcal{F}) \mapsto (\Infl\mathcal{T}, \Infl\mathcal{F}) \quad \text{and} \quad (\mathcal{L}, \mathcal{R}) \mapsto ({\rm Coker}(\mathcal{L}), {\rm Coker}(\mathcal{R}))
\]
give the mutually inverse bijections between the following classes:
\begin{enumerate}
  \item Torsion pairs in $\mathscr{C}$;
  \item Monomorphism factorization systems in $\mathscr{C}$.
\end{enumerate}
\end{corollary}

\begin{corollary}
Let $\mathscr{C}$ be an abelian category. Then assignments
\[
(\mathcal{T}, \mathcal{F}) \mapsto (\Defl\mathcal{T}, \Defl\mathcal{F}) \quad \text{and} \quad (\mathcal{L}, \mathcal{R}) \mapsto ({\rm Ker}(\mathcal{L}), {\rm Ker}(\mathcal{R}))
\]
give the mutually inverse bijections between the following classes:
\begin{enumerate}
  \item Torsion pairs in $\mathscr{C}$;
  \item Epimorphism factorization systems in $\mathscr{C}$.
\end{enumerate}
\end{corollary}

\begin{definition}
Let $\mathscr{C}$ be a triangulated category with the shift functor $[1]$.
A pair $(\mathcal {U},\mathcal {V})$ of subcategories of $\mathscr{C}$ is
called a {\em $t$-structure} in $\mathscr{C}$ if it satisfies the following conditions.

$(1)$~$\mathscr{C}=\mathcal {U}\ast\mathcal {V}$.
$(2)$~$\mathscr{C}(\mathcal {U},\mathcal {V})=0$.
$(3)$~$\mathcal {U}[1]\subseteq\mathcal {U}$.
\end{definition}
When $\mathscr{C}$ is a triangulated category, by \cite[Lemma 3.3]{2}, $t$-structures are exactly
$s$-torsion pairs in $\mathscr{C}$.

\begin{corollary}
Let $\mathscr{C}$ be a triangulated category. Then the assignments
\[
(\mathcal{T}, \mathcal{F}) \mapsto (\Infl\mathcal{T},\Infl\mathcal{F}) \quad \text{and} \quad (\mathcal{L}, \mathcal{R}) \mapsto ({\rm Cone}(\mathcal{L}), {\rm Cone}(\mathcal{R}))
\]
give the mutually inverse bijections between the following classes:
\begin{enumerate}
  \item $t$-structures in $\mathscr{C}$;
  \item Inflation factorization systems in $\mathscr{C}$.
\end{enumerate}
\end{corollary}

\begin{corollary}
Let $\mathscr{C}$ be a triangulated category. Then the assignments
\[
(\mathcal{T}, \mathcal{F}) \mapsto (\Defl\mathcal{T},\Defl\mathcal{F}) \quad \text{and} \quad (\mathcal{L}, \mathcal{R}) \mapsto ({\rm Cocone}(\mathcal{L}), {\rm Cocone}(\mathcal{R}))
\]
give the mutually inverse bijections between the following classes:
\begin{enumerate}
  \item $t$-structures in $\mathscr{C}$;
  \item Deflation factorization systems in $\mathscr{C}$.
\end{enumerate}
\end{corollary}

In what follows, we use an $s$-torsion pair induced by an $(m+1)$-term silting complex to construct an extriangulated factorization system in the $m$-extended module category $m$-${\rm mod}A$ of a finite-dimensional algebra $A$.
According to \cite{1}, for an $(m+1)$-term silting complex ${\bf P}$ in $K^b({\rm proj}A)$, set
\begin{align*}
D^{\leq0}({\bf P})=\{{\bf X}\in D^b({\rm mod}A)~|~\Hom({\bf P},{\bf X}[i])=0,\forall i>0\}
\end{align*}
and
\begin{align*}
D^{\geq0}({\bf P})=\{{\bf X}\in D^b({\rm mod}A)~|~\Hom({\bf P},{\bf X}[i])=0,\forall i<0\}.
\end{align*}
Define
$\mathcal {T}({\bf P})=D^{\leq0}({\bf P})\cap m$-${\rm mod} A$ and $\mathcal {F}({\bf P})=D^{\geq0}({\bf P})[-1]\cap m$-${\rm mod} A$.
Then by \cite[Proposition 2.10]{1}, $(\mathcal {T}({\bf P}),\mathcal {F}({\bf P}))$ is an $s$-torsion pair in $m$-${\rm mod}A$.

\begin{example}
Let $A$ be the path algebra of the quiver $\xymatrix{1 \ar[r]&2 \ar[r]&3.}$ Then the Auslander-Reiten quiver of the 2-extended module category 2-${\rm mod}\,A$ is as follows:
$$
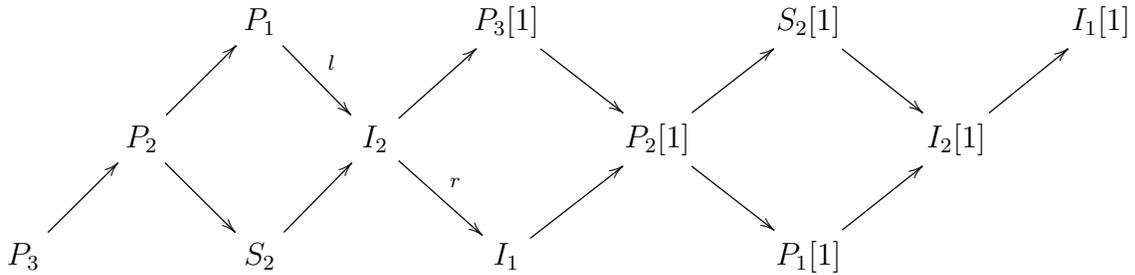
\xymatrix{
  &&P_1 \ar[dr]^-{l} &&P_3[1] \ar[dr] &&S_2[1] \ar[dr]&&I_1[1]\\
&P_2 \ar[ur] \ar[dr] &&I_2 \ar[ur] \ar[dr]^-{r}&&P_2[1] \ar[ur] \ar[dr]&&I_2[1] \ar[ur] & \\
P_3 \ar[ur]&&S_2 \ar[ur]&&I_1 \ar[ur]&&P_1[1] \ar[ur] &&\\
}$$

Let $\mathbf{P}=P_3[1]\bigoplus P_1[1]\bigoplus I_1[1]$, viewed as $3$-term complex. It is straightforward to check that $\mathbf{P}$ is $3$-term silting. Moreover, it is easy to see that
$$ \mathcal{T(\mathbf{P})}={\rm add} \{P_3[1],P_1[1],I_2[1],I_1[1] \}$$
and
$$ \mathcal{F(\mathbf{P})}={\rm add} \{P_3,P_2 ,P_1 ,S_2 ,I_2,I_1,S_2[1] \}.$$

By Theorem \ref{main}, we can obtain an inflation factorization system $(\Infl\mathcal{T(\mathbf{P})},\Infl\mathcal{F(\mathbf{P})})$ in $2$-${\rm mod}A$ with
$\Infl\mathcal{T(\mathbf{P})}={\rm add}\mathcal{X}$ and $\Infl\mathcal{F(\mathbf{P})}={\rm add}\mathcal{Y}$, where
\[
\mathcal{X}=\left\{{\rm id}_X~|~X\in \text{2-}{\rm mod}A\right\}\bigcup\left\{
\begin{aligned}
&0 \rightarrow P_3[1], \quad P_1 \rightarrow I_2, \quad P_2 \rightarrow S_2,\quad 0 \rightarrow I_2[1],\\
&0 \rightarrow P_1[1], \quad I_1 \rightarrow P_2[1], \quad I_2 \rightarrow P_3[1], \\
&0 \rightarrow I_1[1], \quad S_2[1] \rightarrow I_2[1], \quad P_2[1] \rightarrow P_1[1],\\
& P_3[1] \rightarrow P_1[1],\quad I_1 \rightarrow S_2[1],\quad P_2[1] \rightarrow S_2[1]\oplus P_1[1]
\end{aligned}
\right\}
\]
and
\[
\mathcal{Y}=\left\{{\rm id}_Y~|~Y\in \text{2-}{\rm mod}A\right\}\bigcup\left\{
\begin{aligned}
&0 \rightarrow P_3, \quad 0 \rightarrow P_2,\quad 0 \rightarrow P_1,\quad I_2 \rightarrow I_1,\\
&0 \rightarrow S_2, \quad 0 \rightarrow I_2,\quad 0 \rightarrow I_1,\quad S_2 \rightarrow I_2,\\
&P_3 \rightarrow P_1,\quad P_3 \rightarrow P_2,\quad P_2 \rightarrow P_1, \\
&0 \rightarrow S_2[1], \quad P_2 \rightarrow P_1\oplus S_2,\quad P_3[1] \rightarrow P_2[1]
\end{aligned}
\right\}.
\]
The inflation $f=rl$ in the above Auslander-Reiten quiver belongs neither to $\Infl\mathcal{T(\mathbf{P})}$ nor to $\Infl\mathcal{F(\mathbf{P})}$, but $l\in\Infl\mathcal{T(\mathbf{P})}$ and $r\in\Infl\mathcal{F(\mathbf{P})}$, i.e., this is a factorization of $f$ via the inflation factorization system $(\Infl\mathcal{T(\mathbf{P})},\Infl\mathcal{F(\mathbf{P})})$.
\end{example}

\section{Gluing $s$-torsion pairs and factorization systems}

In this section, we consider the gluing of $s$-torsion pairs and extriangulated factorization systems under recollements of extriangulated categories. First of all, let us recall recollements of extriangulated categories.

\begin{definition}\cite[Definition 3.1]{16}
Let $\mathcal{A}$, $\mathcal{B}$ and $\mathcal{C}$ be three extriangulated categories. A \emph{recollement} of $\mathcal{B}$ relative to
$\mathcal{A}$ and $\mathcal{C}$, denoted by ($\mathcal{A}$, $\mathcal{B}$, $\mathcal{C}$), is a diagram
\begin{equation}
  \xymatrix{\mathcal{A}\ar[rr]|{i_{*}}&&\ar@/_1pc/[ll]|{i^{*}}\ar@/^1pc/[ll]|{i^{!}}\mathcal{B}
\ar[rr]|{j^{\ast}}&&\ar@/_1pc/[ll]|{j_{!}}\ar@/^1pc/[ll]|{j_{\ast}}\mathcal{C}}
\end{equation}
given by two exact functors $i_{*},j^{\ast}$, two right exact functors $i^{\ast}$, $j_!$ and two left exact functors $i^{!}$, $j_\ast$, which satisfies the following conditions:
\begin{itemize}
  \item [(R1)] $(i^{*}, i_{\ast}, i^{!})$ and $(j_!, j^\ast, j_\ast)$ are adjoint triples.
  \item [(R2)] ${\rm Im} i_{\ast}={\rm Ker} j^{\ast}$.
  \item [(R3)] $i_\ast$, $j_!$ and $j_\ast$ are fully faithful.
  \item [(R4)] For each $X\in\mathcal{B}$, there exists a left exact $\mathbb{E}_{\mathcal {B}}$-triangle sequence
  \begin{equation}
  \xymatrix{i_\ast i^! X\ar[r]^-{\theta_X}&X\ar[r]^-{\vartheta_X}&j_\ast j^\ast X\ar[r]&i_\ast A}
   \end{equation}
  with $A\in \mathcal{A}$, where $\theta_X$ and  $\vartheta_X$ are given by the adjunction morphisms.
  \item [(R5)] For each $X\in\mathcal{B}$, there exists a right exact $\mathbb{E}_{\mathcal {B}}$-triangle sequence
  \begin{equation}
  \xymatrix{i_\ast\ar[r] A' &j_! j^\ast X\ar[r]^-{\upsilon_X}&X\ar[r]^-{\nu_X}&i_\ast i^\ast X&}
   \end{equation}
 with $A'\in \mathcal{A}$, where $\upsilon_X$ and $\nu_X$ are given by the adjunction morphisms.
\end{itemize}
\end{definition}
See \cite[Lemma 3.3, Proposition 3.4]{16} for some properties of recollements of extriangulated categories.
\begin{definition} An extriangulated category $\mathscr{C}$ is called an {\em extriangulated category with balanced negative first extensions}, if $\mathscr{C}$ has a negative first extension structure which satisfies that
$$\mathbb{E}^{-1}(-, \mathcal{P}(\mathscr{C}))=0.$$
\end{definition}
Clearly, exact categories and triangulated categories are extriangulated
categories with balanced negative first extensions.
\begin{lemma}\label{Iso}
Let $\mathcal{A}, \mathcal{B}$ be extriangulated categories with balanced negative first
extensions and $\mathcal{A}$ has enough projectives. Let $G:\mathcal{A}\rightarrow \mathcal{B}$ be a functor which admits a left adjoint functor $F$. If $G$ is exact and preserves projectives,
then $\mathbb{E}^{-1}_{\mathcal{A}}(FX,Y)\cong \mathbb{E}^{-1}_{\mathcal{B}}(X,GY)$ for any $X\in\mathcal{B}$ and $Y\in\mathcal{A}$.
\end{lemma}
\begin{proof} For any $Y\in \mathcal{A}$, since $\mathcal{A}$ has enough projectives, there exists an $\mathbb{E}_{\A}$-triangle \begin{equation}\label{1}\xymatrix{Y'\ar[r]^-{}&P\ar[r]^-{}&Y\ar@{-->}[r]^{}&}\end{equation} with $P\in \mathcal{P({\mathcal{A}})}$. Since $G$ is an exact functor and preserves projectives, we obtain the following $\mathbb{E}_{\mathcal{B}}$-triangle
\begin{equation}\label{2}\xymatrix{GY'\ar[r]^-{}&GP\ar[r]^-{}&GY\ar@{-->}[r]&}\end{equation}
with $GP\in \mathcal{P}(\mathcal{B})$. Applying the functors $\mathcal{A}(FX,-)$ and $\mathcal {B}(X,-)$ to (\ref{1}) and (\ref{2}), respectively, we get the following commutative diagram of exact sequences
$$\xymatrix{
  0\ar[r]&\mathbb{E}^{-1}_{\mathcal{A}}(FX,Y)\ar[r]\ar[d]&\mathcal{A}(FX,Y')\ar[d]^-{\cong} \ar[r] & \mathcal{A}(FX,P)  \ar[d]^-{\cong}\\
 0\ar[r]&\mathbb{E}^{-1}_{\mathcal{B}}(X,GY)\ar[r] & \mathcal{B}(X,GY')\ar[r] &  \mathcal{B}(X,GP) .}$$
By Five-Lemma, we finish the proof.
\end{proof}

\begin{theorem}\label{gs}
 Let $(\mathcal{A}, \mathcal{B}, \mathcal{C})$ be a recollement of extriangulated categories with balanced negative first extensions. Let $(\mathcal{T}_{1}, \mathcal{F}_{1})$ and $(\mathcal{T}_{2}, \mathcal{F}_{2})$ be $s$-torsion pairs in $\mathcal{A}$ and $\mathcal{C}$, respectively.
  Set
  $$\mathcal{T}=\{B\in\mathcal{B}\;|\;i^*B\in\mathcal{T}_{1}~\text{and}~j^*B\in\mathcal{T}_{2}\}~\text{and}~
  \mathcal{F}=\{B\in\mathcal{B}\;|\;i^!B\in\mathcal{F}_{1}~\text{and}~j^*B\in\mathcal{F}_{2}\}.$$
 If $i^*,i^!$ are exact and $i^!$ preserves projective objects, then $(\mathcal{T}, \mathcal{F})$ is an $s$-torsion in $\mathcal{B}$.
\end{theorem}

\begin{proof}
For any $T\in\mathcal{T}$ and $F\in\mathcal{F}$, since $i^*$ is exact, by \cite[Proposition 3.4(2)]{16}, there is an $\mathbb{E}_{\mathcal{B}}$-triangle
  $$j_!j^{\ast}T{\longrightarrow}
  T{\longrightarrow}i_{\ast}i^{\ast}T\dashrightarrow.$$
  Applying $\mathcal{B}(-, F)$ to this $\mathbb{E}_{\mathcal{B}}$-triangle, we get an exact sequence
  $$\xymatrix@C=.18in{
 \mathbb{E}^{-1}_{\mathcal{B}}(i_*i^*T,F) \ar[r]&\mathbb{E}^{-1}_{\mathcal{B}}(T,F) \ar[r]&\mathbb{E}^{-1}_{\mathcal{B}}(j_!j^*T,F)\ar[r]&\mathcal{B}(i_*i^*T,F) \ar[r]&\mathcal{B}(T,F) \ar[r]&\mathcal{B}(j_!j^*T,F).
  }$$
  Note that $i^*T\in \mathcal{T}_{1}$, $i^!F\in \mathcal{F}_{1}$, $j^*T\in \mathcal{T}_{2}$ and $j^*F\in \mathcal{F}_{2}$, we have that
  $$\mathcal{B}(i_*i^*T,F)\cong \mathcal{B}(i^*T,i^!F)=0~\text{and}~
  \mathcal{B}(j_!j^*T,F)\cong\mathcal{B}(j^*T,j^*F)=0.$$ Thus, $\mathcal{B}(T,F)=0$, i.e., $\mathcal{B}(\mathcal{T}, \mathcal{F}) =0$.

  Since $i^!$ is exact and preserves projectives, by Lemma \ref{Iso},
  we obtain that $$\mathbb{E}^{-1}_{\mathcal{B}}(i_*i^*T,F)\cong \mathbb{E}^{-1}_{\mathcal{A}}(i^*T,i^!F)=0.$$
  Similarly, since $i^!$ is exact, by \cite[Lemma 3.3(8')]{16}, we obtain that $j_*$ is exact, and then $j^*$ preserves projectives by \cite[Lemma 3.3(4)]{16}. Since $j^*$ is also exact, by Lemma \ref{Iso}, we obtain that  $$\mathbb{E}^{-1}_{\mathcal{B}}(j_!j^*T,F)\cong \mathbb{E}^{-1}_{\mathcal{C}}(j^*T,j^*F)=0.$$ Hence, $\mathbb{E}^{-1}(T, F) =0$, i.e., $\mathbb{E}^{-1}(\mathcal{T}, \mathcal{F}) =0$. The proof of $\mathcal{B}=\mathcal{T}\ast\mathcal{F}$ is the same as given in the proof of \cite[Theorem 3.4]{hhp}. Therefore, we complete the proof.
\end{proof}

\begin{theorem}
  Let $(\mathcal{A}, \mathcal{B}, \mathcal{C})$ be a recollement of extriangulated categories with balanced negative extensions. Let $(\mathcal{E}_{1}, \mathcal{M}_{1})$ and $(\mathcal{E}_{2}, \mathcal{M}_{2})$ be inflation factorization systems in $\mathcal{A}$ and $\mathcal{C}$, respectively.
  Set
  $$\mathcal{E}=\{f\;{\rm is}\;{\rm an}\;{\rm  inflation}\;|\;i^*f\;\in \mathcal{E}_{1}\; {\rm and}\; j^*f\; \in \mathcal{E}_{2}\}$$
  and
$$\mathcal{M}=\{g\;{\rm is}\;{\rm an}\;{\rm  inflation}\;|\;i^!g\;\in \mathcal{M}_{1}\; {\rm and}\; j^*g\; \in \mathcal{M}_{2}\}.$$
 If $i^*,i^!$ are exact and $i^!$ preserves projective objects, then $(\mathcal{E}, \mathcal{M})$ is an inflation factorization system in $\mathcal{B}$.
\end{theorem}
\begin{proof}
By Theorem \ref{main}, $(\Cone\mathcal{E}_{1},\Cone\mathcal{M}_{1})$ and $(\Cone\mathcal{E}_{2},\Cone\mathcal{M}_{2})$ are $s$-torsion pairs in $\mathcal{A}$ and $\mathcal{C}$, respectively. Set
$$\mathcal{T}=\{B\in\mathcal{B}\;|\;i^*B\in\Cone\mathcal{E}_{1}~\text{and}~j^*B\in\Cone\mathcal{E}_{2}\}$$
and
$$\mathcal{F}=\{B\in\mathcal{B}\;|\;i^!B\in\Cone\mathcal{M}_{1}~\text{and}~j^*B\in\Cone\mathcal{M}_{2}\}.$$
By Theorem \ref{gs}, $(\mathcal{T},\mathcal{F})$ is an $s$-torsion pair in $\mathcal{B}$. Then
\begin{align*}\Infl\mathcal {T}&=\{f~\text{is~an~inflation}~|~{\rm cone}(f)\in\mathcal {T}\}\\
&=\{f~\text{is~an~inflation}~|~i^*({\rm cone}(f))\in\Cone\mathcal{E}_{1}~\text{and}~~j^*({\rm cone}(f))\in\Cone\mathcal{E}_{2}\}\\
&=\{f~\text{is~an~inflation}~|~{\rm cone}(i^*f)\in\Cone\mathcal{E}_{1}~\text{and}~~{\rm cone}(j^*f)\in\Cone\mathcal{E}_{2}\}=:\mathcal{E}'.
\end{align*}
Clearly, $\mathcal{E}\subseteq\mathcal{E}'$. Conversely, for any $f\in\mathcal{E}'$, since $i^*, j^*$ are exact, we have that $i^*f\in\Infl(\Cone\mathcal{E}_1)=\mathcal{E}_1$ and $j^*f\in\Infl(\Cone\mathcal{E}_2)=\mathcal{E}_2$, i.e., $\mathcal{E}'\subseteq\mathcal{E}$. Hence, $\Infl\mathcal{T}=\mathcal{E}$.

Similarly, we have that $\Infl\mathcal {F}=\mathcal{M}$. By Theorem \ref{main} again, we obtain that $(\mathcal{E}, \mathcal{M})$ is an inflation factorization system in $\mathcal{B}$.
\end{proof}

Dually, we have the following theorem, and omit the proof.
\begin{theorem}
  Let $(\mathcal{A}, \mathcal{B}, \mathcal{C})$ be a recollement of extriangulated categories with balanced negative extensions. Let $(\mathcal{E}_{1}, \mathcal{M}_{1})$ and $(\mathcal{E}_{2}, \mathcal{M}_{2})$ be deflation factorization systems in $\mathcal{A}$ and $\mathcal{C}$, respectively. Set
  $$\mathcal{E}=\{f\;{\rm is}\;{\rm a}\;{\rm  deflation}\;|\;i^*f\;\in \mathcal{E}_{1}\; {\rm and}\; j^*f\; \in \mathcal{E}_{2}\}$$
and
$$\mathcal{M}=\{g\;{\rm is}\;{\rm a}\;{\rm  deflation}\;|\;i^!g\;\in \mathcal{M}_{1}\; {\rm and}\; j^*g\; \in \mathcal{M}_{2}\}.$$
 If $i^*,i^!$ are exact and $i^!$ preserves projective objects, then $(\mathcal{E}, \mathcal{M})$ is a deflation factorization system in $\mathcal{B}$.
\end{theorem}

\section*{Acknowledgments}
	The work is partially supported by the National Natural Science Foundation of China (No. 12271257) and the Natural Science Foundation of Jiangsu Province of China (No. BK20240137).
%The authors are very grateful to the anonymous referee for the careful reading and helpful suggestions, which have made the paper more readable.

\end{document}